\documentclass[12pt]{amsart}

\usepackage{enumerate}
\usepackage{color}
\usepackage{xypic}
\newtheorem{theorem}{Theorem}[section]
\newtheorem{lemma}[theorem]{Lemma}

\newtheorem{corollary}[theorem]{Corollary}

\theoremstyle{definition}

\theoremstyle{remark}
\newtheorem{remark}[theorem]{Remark}

\numberwithin{equation}{section}

\begin{document}

\title[A maximum principle related to volume growth]{A maximum principle related to volume growth and applications} 

\author{L. J. Al\'\i as}
\address{Departamento de Matem\'aticas, Universidad de Murcia, E-30100 Espinardo, Murcia, Spain}
\email{ljalias@um.es}
\thanks{This research is a result of the activity developed within the framework of the Programme in Support of Excellence Groups of the Regi\'on de Murcia, Spain, by Fundaci\'on S\'eneca, Science and Technology Agency of the Regi\'on de Murcia.\\
\indent Luis J. Al\'\i as was partially supported
by MICINN/FEDER project PGC2018-097046-B-I00 and Fundaci\'on S\'eneca project 19901/GERM/15, Spain.\\
\indent Antonio Caminha was partially supported by PRONEX/FUNCAP/CNPQ-PR2-0101-00089.01.00/15
and Est\'{\i}mulo \`a Coopera\c c\~ao Cient\'{\i}fica e Desenvolvimento da P\'{o}s-Gradua\c c\~ao Funcap/CAPES
Projeto nº 88887.165862/2018-00, Brazil.}

\author{A. Caminha}
\address{Departamento de Matem\'atica, Universidade Federal do Cear\'a,  Campus do Pici, 60455-760 Fortaleza-Ce, Brazil}
\email{caminha@mat.ufc.br}
%\thanks{}

%    Information for third author
\author{F. Y. do Nascimento}
\address{Universidade Federal do Cear\'a, Crate\'us, BR 226, Km 4, Cear\'a, Brazil}
\email{yure@ufc.br}

%    General info

\begin{abstract}
In this paper, we derive a new form of maximum principle for {\color{black} smooth functions on a complete noncompact Riemannian manifold $M$ for which there exists a bounded vector field $X$ such that $\langle\nabla f,X\rangle\geq 0$ on $M$ and $\mathrm{div} X\geq af$ outside a suitable compact subset} of $M$, for some constant $a>0$, under the assumption that $M$ has either polynomial or exponential volume growth. We then use it to obtain some straightforward applications to smooth functions and, more interestingly, to Bernstein-type results for hypersurfaces immersed into a Riemannian manifold endowed with a Killing vector field, as well as to some results on the existence and size of minimal submanifolds immersed into a Riemannian manifold endowed with a conformal vector field. 
\end{abstract}

\maketitle

\section{Introduction}

Maximum principles appear naturally in Differential Geometry, due to the fact that many different geometric situations are analytically modeled by certain linear or quasilinear elliptic partial differential operators, for which several versions of maximum principles play a key role in the theory. In a recent paper of us \cite{Alias:19}, we derived a new form of maximum principle which is appropriate for controlling the behavior of a smooth vector field with nonnegative divergence on a complete noncompact Riemannian manifold, and which is the analogue of the simple fact that, on such a manifold, a nonnegative subharmonic function that vanishes at infinity actually vanishes identically (Theorem 2.2 in \cite{Alias:19}).

In this paper we derive {\color{black} a maximum principle for smooth functions $f$ on a complete noncompact Riemannian manifold $M$, assuming that there exists a bounded vector field $X$ on $M$ such that $\langle\nabla f,X\rangle\geq 0$ on $M$ and $\mathrm{div} X\geq af$ outside a suitable compact subset $K$ of $M$, for some constant $a>0$, under the assumption that} $M$ has either polynomial or exponential volume growth (see Theorem \ref{prop:maximum principle II} below for the precise statement of the result). As first consequences of this new maximum principle, we obtain some straightforward applications to {\color{black} solutions and subsolutions of some PDE on $M$}, including an extension of a classical result of Cheng and Yau \cite{Cheng:76}  (see Corollary \ref{cor:Cheng-Yau for exponent 1}), as well as a result somewhat related to a classical result of Fischer-Colbrie and Schoen \cite{FischerColbrie:80}
(see Corollary \ref{coro:Fischer Colbrie-Schoen 1}). 

In Section \ref{section3} we present some interesting applications of our maximum principle to Bernstein-type results for hypersurfaces immersed into a Riemannian manifold endowed with a Killing vector field, allowing us to extend some of the results of \cite{Alias:19} to the case of bounded second fundamental form, replacing the behavior of  the Gauss map of the hypersurface at infinity by {\color{black} an estimate} on the size of the support function {\color{black} on $M$}. See, for instance, Theorem \ref{thm:Bernstein for CMC hypersurfaces}, and its corollaries \ref{coro:Bernstein for CMC hypersurfaces of Lie groups} and \ref{coro:Bernstein for CMC hypersurfaces of Rn},  for the case of constant mean curvature hypersurfaces, and Theorem \ref{thm:Bernstein for Sr constant hypersurfaces} for its generalization to the case of constant higher order mean curvature. Finally, in Section \ref{section4} we apply our maximum principle to some results on the existence and size of minimal submanifolds immersed into a Riemannian manifold endowed with a conformal vector field (Theorem \ref{thm:general on submersions}) and, in particular, into Riemannian warped products (Corollary \ref{cor:specializing to warped products}). {\color{black} Yet more particularly}, we prove, among other results, that there exists no complete {\color{black} noncompact} minimal submanifold with image {\color{black} contained in} an Euclidean ball and having polynomial volume growth (item (a) in Corollary \ref{cor:specializing to Rm}), and that the same happens for complete {\color{black} noncompact} minimal submanifolds into the hyperbolic space with image contained in the open half space bounded by a horosphere  (item (a) in Corollary \ref{cor:specializing to Hm}).
 
\section{The maximum principle}

Let $M$ be a connected, oriented, complete noncompact Riemannian manifold. We denote by $B(p,t)$ the geodesic ball centered at $p$ and with radius $t$. 

Given a continuous function $\sigma:(0,+\infty)\to(0,+\infty)$, we say that {\em $M$ has volume growth like $\sigma(t)$} if there exists $p\in M$ such that 
%$$\lim_{t\to+\infty}\frac{\mathrm{vol}(B(p,t))}{\sigma(t)}=c(p),$$ 
$$\mathrm{vol}(B(p,t))=O(\sigma(t))$$ 
as $t\to+\infty$, where $\mathrm{vol}$ denotes the Riemannian volume. 

If $p,q\in M$ are at distance $d$ from each other, it is straightforward to check that
$$\frac{\mathrm{vol}(B(p,t))}{\sigma(t)}\geq\frac{\mathrm{vol}(B(q,t-d))}{\sigma(t-d)}\cdot\frac{\sigma(t-d)}{\sigma(t)}.$$
%Therefore, if $\mathrm{vol}(B(p,t))=O(\sigma(t))$ as $t\to+\infty$, then the same holds at $q$ as $t\to+\infty$. 
Hence, the choice of $p$ in the notion of volume growth is immaterial, so that, henceforth, we shall simply say that $M$ has {\em polynomial} (resp. {\em exponential}) {\em volume growth}, according to the case.

{\color{black}
For the statement of the coming result, we also recall that a vector field $X$ on a complete Riemannian manifold $M$ is {\em complete} provided its flow $\{\psi_t\}$ is globally defined, and that this is always the case if $X$ is bounded on $M$. Moreover, a subset $\Omega$ of $M$ is {\em stable under the flow of $X$} if $\psi_t(\Omega)\subset\Omega$ for every $t\geq 0$. In particular, this also holds for}  {\color{black} $\Omega=M$.}

\begin{theorem}\label{prop:maximum principle II}
Let $M$ be a connected, oriented, complete noncompact Riemannian manifold, {\color{black} let $X\in\mathfrak{X}(M)$ be a bounded vector field on $M$, with $|X|<c$, and $K$ be a {\color{black} (possibly empty)} compact subset of $M$ such that $M\setminus K$ is stable under the flow of $X$}. Assume that {\color{black} $f\in\mathcal C^{\infty}(M)$ is such that} $\langle\nabla f,X\rangle\geq 0$ on $M$ and $\mathrm{div} X\geq af$ on $M\setminus K$, for some $a>0$.
\begin{enumerate}[$(a)$]
\item If $M$ has polynomial volume growth, then $f\leq 0$ on $M\setminus K$.
\item If $M$ has exponential volume growth, say like $e^{\beta t}$, then $f\leq\frac{c\beta}{a}$ on $M\setminus K$.
\end{enumerate}
\end{theorem}

\begin{proof} 
Suppose that there is a $p\in M\setminus K$ such that $f(p)>0$, and choose $\alpha$ and $r$ satisfying $0<\alpha<f(p)$ and $B=B(p,r)\subset\subset A_{\alpha} = \{x\in M\setminus K;f(x) > \alpha \}$.
	
Since $|X|$ is bounded, the flow $\psi_t$ of $X$ is defined for every $t \in \mathbb{R}$, whence we can define the smooth function $\varphi:[0,+\infty)\to(0,+\infty)$ by letting
$$\varphi(t)=\mathrm{vol}(\psi_{t}(B))=\int_{\psi_t(B)}dM=\int_{B}\psi_t^*dM.$$

Since $\overline B$ is compact, we can differentiate under the integral sign to obtain
\begin{equation}\label{eq:derivative of phi}
\begin{split}
\varphi '(t_0)&=\left.\dfrac{d}{dt}\right|_{t=0} \int_{\psi_{t+t_0}(B)}dM = \left.\dfrac{d}{dt}\right|_{t=0} \int_{\psi_{t_0}(B)}\psi^{\ast}_t(dM) \\
&=\int_{\psi_{t_0}(B)}\left.\dfrac{d}{dt}\right|_{t=0}\psi^{\ast}_t(dM) = \int_{\psi_{t_0}(B)}\mathrm{div} X \,dM.
\end{split}
\end{equation}
	
Now,
$$\frac{d}{dt}f(\psi_t(x))=\langle\nabla f,X\rangle_{\psi_t(x)}\geq 0,$$
so that, for $x \in A_{\alpha}$, we have 
$$f(\psi_t(x))\geq f(\psi_0(x))=f(x)>\alpha,\ \ \forall\,\,t\geq 0.$$
{\color{black} Since $M\setminus K$ is stable under the flow of $X$, we thus get} $\psi_t(A_{\alpha})\subset A_{\alpha}$, for all $t\geq 0$.
	
The inequality $\mathrm{div} X\geq af$ on $M\setminus K$, together with \eqref{eq:derivative of phi} and the fact that $\psi_t(B) \subset\psi_t(A_{\alpha})\subset A_{\alpha}\subset M\setminus K$, then give 
\begin{equation}
\varphi'(t)\geq\int_{\psi_{t}(B)}af \,dM >a\alpha \int_{\psi_{t}(B)}\,dM =a\alpha\varphi(t)
\end{equation}
for all $t\geq 0$. In particular, $\varphi'(t)>0$ for all $t\geq 0$, whence $\varphi(t)\geq\varphi(0)=\mathrm{vol}(B)>0$ for all $t\geq 0$. Integrating the inequality $\frac{\varphi'(s)}{\varphi(s)}>a\alpha$ along the interval $[0,t]$, we obtain 
\begin{equation}\label{eq:phi greater than exp}
\varphi(t)>\mathrm{vol}(B)e^{a\alpha t},\ \ \forall\,\,t>0.
\end{equation}
	
Let $d(x,\psi_t(x))$ denote the Riemannian distance between $x$ and $\psi_t(x)$. Since $|X| \leq c$, we get
\begin{equation}
d(x,\psi_t(x)) \leq \int_0^t \left|\frac{d}{ds}\psi_s(x)\right|ds = \int_0^t |X(\psi_s(x))|ds \leq ct.
\end{equation}
On the other hand, for every $x\in B$ we have 
$$d(p,\psi_t(x)) \leq d(x,\psi_t(x)) + d(p,x) < ct + r.$$
In turn, this means that $\psi_t(B)\subset B(p,ct + r)$ for every $t\geq 0$, and it follows from \eqref{eq:phi greater than exp} that 
$$\mathrm{vol}(B(p,ct+r))\geq \mathrm{vol}(\psi_t(B))=\varphi(t)>\mathrm{vol}(B)e^{a\alpha t},\ \ \forall\,\,t\geq 0.$$
A linear change of variables thus gives some constant $C>0$ such that
\begin{equation}\label{eq:volume growth greater than exp}
\mathrm{vol}(B(p,t))>Ce^{\frac{a\alpha}{c}t},\ \ \forall\,\,t\geq r.
\end{equation}
	
If $M$ has polynomial volume growth, then \eqref{eq:volume growth greater than exp} cannot be true for every $t\geq r$. Thus, our initial supposition that $f(p)>0$ for some $p\in M\setminus K$ leads to a contradiction, whence $f\leq 0$ on $M\setminus K$.

On the other hand, if $M$ has exponential volume growth, say like $e^{\beta t}$, and there exists $p\in M\setminus K$ such that $f(p)>\frac{c\beta}{a}$, then we start the previous reasoning by choosing the real number $\alpha$ satisfying $\frac{c\beta}{a}<\alpha<f(p)$. Therefore \eqref{eq:volume growth greater than exp} cannot be true for every $t\geq r$, which is a contradiction. Hence, $f\leq\frac{c\beta}{a}$ on $M\setminus K$.
\end{proof}

Theorem \ref{prop:maximum principle II} has a number of straightforward applications to smo\-oth functions on Riemannian manifolds, and we collect them in the sequel. The first {\color{black} two ones extend} a classical result of Cheng and Yau (cf. \cite{Cheng:76} or \cite{Caminha:06}).

\begin{corollary}\label{cor:Cheng-Yau for exponent 1}
Let $M$ be a connected, oriented, complete noncompact Riemannian manifold {\color{black} and $\Omega\subset M$ be a bounded domain. Let $f\in\mathcal C^{\infty}(M)$ be a nonnegative function satisfying the following conditions:
\begin{enumerate}[$(a)$]
\item $\Delta f\geq af$ on $M\setminus\Omega$, for some $a>0$. 
\item $|\nabla f|$ is bounded on $M$.
\item $M\setminus\Omega$ is stable under the flow of $\nabla f$. 
\end{enumerate}
If $M$ has polynomial volume growth, then $f\equiv 0$ on $M\setminus\Omega$.}
\end{corollary}

\begin{proof} 
Taking $X=\nabla f$ and $K=\overline\Omega$ in the previous {\color{black} result}, we conclude that $f\leq 0$ in $M\setminus\Omega$. However, since $f\geq 0$ on $M$, we get $f\equiv 0$ on $M\setminus\Omega$.
\end{proof}

{\color{black}
\begin{corollary}\label{cor:Cheng-Yau for exponent 2}
Let $M$ be a connected, oriented, complete noncompact Riemannian manifold, and let $f\in\mathcal C^{\infty}(M)$ be a nonnegative function such that $\Delta f\geq af$ on $M$, {\color{black} for some} $a>0$. If $M$ has polynomial volume growth and $|\nabla f|$ is bounded on $M$, then $f\equiv 0$ {\color{black} on $M$}.
\end{corollary}

\begin{proof} 
Just take $\Omega=\emptyset$ in the previous corollary.
\end{proof}
}

{\color{black} 
\begin{remark}
The hypothesis on the stability of $M\setminus\Omega$ under the flow of $\nabla f$ is necessary. Indeed, with $M=\mathbb R^2$ one could take a smooth $f$ such that $f(x) = K_n(|x|)$ on $|x|>1$, where $K_n$ is the $n$-th modified Bessel function of the second kind, {\color{black} $n=0,1,\ldots$}. This function $f$ is strictly positive, bounded, with bounded gradient and satisfies the inequality $\Delta f\geq f$ on $\mathbb R^2\setminus B(0;1)$.  {\color{black} Indeed, by a direct computation one has
\[
\nabla f(x)=K_n'(|x|)\frac{x}{|x|} \quad \text{and} \quad 
\Delta f(x)=K''_n(|x|)+\frac{K_n'(|x|)}{|x|}
\]
outside $B(0;1)$, where the function $K_n(t)>0$ satisfies the modified Bessel differential equation
\[
t^2y''(t)+ty'(t)-(t^2+n^2)y(t)=0.
\]
It easily follows from here that, outside $B(0;1)$,
\[
\Delta f(x)=\left(1+\frac{n^2}{|x|^2}\right)K_n(|x|)\geq f(x).
\]
}
Since $\nabla f(x)=K_n'(|x|)\frac{x}{|x|}$ outside $B(0;1)$ and $K_n'<0$, we note that $\mathbb R^2\setminus B(0;1)$ is not stable under the flow of $\nabla f$. We would like to thank professors S. Pigola and A. Setti for calling our attention to this example.
\end{remark}
}

The next {\color{black} two applications are} somewhat related to a classical result of Fischer-Colbrie and Schoen (cf. \cite{FischerColbrie:80}).

\begin{corollary}\label{coro:Fischer Colbrie-Schoen 1}
Let $M$ be a connected, oriented, complete noncompact Riemannian manifold, {\color{black} $\Omega\subset M$ be a bounded domain and} $q(x)\in\mathcal C^{\infty}(M)$ be a positive function such that $\inf_{x\in M}q(x)>0$. If $M$ has polynomial volume growth, then the operator $\Delta-q(x)$ has no nontrivial solution $f\in C^{\infty}(M)$ satisfying the following conditions:
\begin{enumerate}[$(a)$]
\item $f$ is nonnegative {\color{black} outside $\Omega$}.
\item $\nabla f$ is bounded on $M$.
\item {\color{black} $M\setminus\Omega$ is stable under the flow of $\nabla f$.} 
\end{enumerate}
\end{corollary}

\begin{proof} 
{\color{black} Assume that there exists $f\in C^{\infty}(M)$ satisfying the stated conditions}. Letting $a=\inf_Mq>0$, we get $\Delta f\geq af$ on $M\setminus\Omega$, and {\color{black} Corollary \ref{cor:Cheng-Yau for exponent 1}} gives $f\equiv 0$ on $M\setminus\Omega$. Therefore, unique continuation for solutions of elliptic PDEs of second order (cf. \cite{Aronszajn:57}) gives $f\equiv 0$ on $M$.
\end{proof}

{\color{black}
\begin{corollary}\label{coro:Fischer Colbrie-Schoen 2}
Let $M$ be a connected, oriented, complete noncompact Riemannian manifold, and $q(x)\in\mathcal C^{\infty}(M)$ be a positive function such that $\inf_{x\in M}q(x)>0$. If $M$ has polynomial volume growth, then the operator $\Delta-q(x)$ has no  nontrivial solution $f\in C^{\infty}(M)$ satisfying the following conditions:
\begin{enumerate}[$(a)$]
\item $f$ is nonnegative on $M$.
\item $\nabla f$ is bounded on $M$.
\end{enumerate}
\end{corollary}

\begin{proof} 
Just take $\Omega=\emptyset$ in the proof of the previous corollary.
\end{proof}
}

\begin{remark}
More generally, there are results similar to those of {\color{black} corollaries \ref{cor:Cheng-Yau for exponent 1} to \ref{coro:Fischer Colbrie-Schoen 2}} for divergence-type second order elliptic partial differential operators of the form $Lu=\mathrm{div}(T(\nabla u))$, where $T:\mathfrak{X}(M)\rightarrow\mathfrak{X}(M)$ is a symmetric positive definite $(1,1)$-tensor field on $M$, with $\sup_M\|T\|<+\infty$. The proofs are the same, just taking $X=T(\nabla f)$.
\end{remark}

\section{Bernstein-type results for hypersurfaces}\label{section3}

Let $\overline M^{n+1}$ be an oriented Riemannian manifold endowed with a Killing vector field $Y$. Let also $\varphi:M^n\to\overline M^{n+1}$ be an isometric immersion of a connected, orientable, complete noncompact Riemannian manifold $M^n$ into $\overline M$, and orient $M$ by the choice of a globally defined unit normal vector field $N$. 

In this section, we will apply item (a) of Theorem \ref{prop:maximum principle II} to study the behavior of $\varphi$. Our aim is to extend some of the results of \cite{Alias:19} to the case of bounded second fundamental form, replacing the behavior of $N$ at infinity by a suitable estimate on the size of the support function {\color{black} $\eta=\langle N,Y\rangle$ on $M$}.

We start by computing the gradient $\nabla\eta$ of $\eta$. To this end, we let $A(\cdot)=-\overline\nabla_{(\cdot)}N$ stand for the Weingarten operator of $\varphi$ with respect to $N$, fix a point $p$ on $M$ and a vector $v\in T_pM$. Then, Killing's equation, together with the symmetry of $A$, give at $p$
\begin{equation}\nonumber
\begin{split}
\langle\nabla\eta,v\rangle&\,=v(\eta)=\langle\overline\nabla_vN,Y\rangle+\langle N,\overline\nabla_vY\rangle\\
&\,=-\langle A_pv,Y\rangle-\langle\overline\nabla_NY,v\rangle\\
&\,=\langle-A_pY^{\top}-\overline\nabla_NY,v\rangle,
\end{split}
\end{equation}
where $Y^{\top}$ denotes the orthogonal projection of $Y_{|M}$ onto $TM$. We now observe, thanks again to Killing's equation, that $\langle\overline\nabla_NY,N\rangle=0$. Therefore, $\overline\nabla_NY$ is tangent to $M$, and the above computation gives
\begin{equation}\label{eq:gradient of the support function of a Killing vector field}
\nabla\eta=-AY^{\top}-\overline\nabla_NY.
\end{equation}

If the Killing vector field $Y$ has unit norm, then Cauchy-Schwarz inequality shows that $\eta\leq 1$. Moreover, equality holds on all of $M$ if and only if $N=Y$ along $M$, in which case $M$ is a leaf of the distribution $\langle Y\rangle^{\bot}$. Also in this case, for $p\in M$ and $u,v\in T_pM$, the Killing condition of $Y$ allows us to compute, at $p$,
$$\langle A_pu,v\rangle=-\langle\overline\nabla_uN,v\rangle=-\langle\overline\nabla_uY,v\rangle=\langle\overline\nabla_vY,u\rangle=-\langle A_pv,u\rangle.$$
Since $A$ is symmetric, this implies $A_p=0$ and, since $p$ was arbitrarily chosen, $M$ is totally geodesic in $\overline M$.

As a final preliminary, we say (cf. \cite{Dillen:09} or \cite{Garnica:12}, for instance) that $Y$ is a {\em canonical direction} for $\varphi$ (or for $M$, whenever $\varphi$ is clear from the context) if $Y^{\top}$ is a principal direction of $A$. 

We can state and prove our first result, which goes as follows.

\begin{theorem}\label{thm:Bernstein for CMC hypersurfaces}
Let $\overline M^{n+1}$ be an oriented Einstein Riemannian manifold endowed with a Killing vector field $Y$ of unit norm. Let $\varphi:M^n\to M^{n+1}$ be an isometric immersion of a connected, orientable, complete noncompact Riemannian manifold $M^n$ into $\overline M$. Orient $M$ by the choice of a globally defined unit normal vector field $N$, and assume that $M$ has cmc and bounded second fundamental form $A$. 
If $M$ has polynomial volume growth, $Y$ is either parallel or a canonical direction for $\varphi$ and the support function $\eta=\langle N,Y\rangle$ satisfies
\begin{equation}\label{eq:growth of the support function}
\eta\geq\frac{1}{|A|^2+1}
\end{equation}
{\color{black} on $M$}, then $M$ is a leaf of the distribution $\langle Y\rangle^{\bot}$. In particular, $M$ is totally geodesic in $\overline M$.
\end{theorem}

\begin{proof}
As before, we let $\eta=\langle N,Y\rangle$ be the support function with respect to $Y$, let $Y^{\top}$ denote the orthogonal projection of $Y_{|M}$ onto $TM$ and set $X=AY^{\top}$ and $f=1-\eta$. 

If $Y$ is parallel, then \eqref{eq:gradient of the support function of a Killing vector field} readily gives $\nabla\eta=-AY^{\top}=-X$. If $Y$ is a canonical direction for $\varphi$, say with $AY^{\top}=\lambda Y^{\top}$, then \eqref{eq:gradient of the support function of a Killing vector field}, together with the fact that $\overline\nabla_NY$ has no orthogonal component and $|Y|=1$, give
\begin{equation}\nonumber
\begin{split}
\langle\nabla\eta,X\rangle&\,=\langle-X-\overline\nabla_NY,X\rangle=-|X|^2-\langle\overline\nabla_NY,\lambda Y^{\top}\rangle\\
&\,=-|X|^2-\lambda\langle\overline\nabla_NY,Y\rangle=-|X|^2.
\end{split}
\end{equation}
Thus, in each of the cases above, we have
$$\langle\nabla f,X\rangle=-\langle\nabla\eta,X\rangle=|X|^2\geq 0.$$

On the other hand, as computed in \cite{Alias:19} for any Killing vector field $Y$,
$$\mathrm{div}_M(X)=-{\rm Ric}_{\overline M}(Y^{\top},N)+Y^{\top}(nH)+\eta|A|^2,$$
where $H=\frac{1}{n}\text{tr}(A)$ stands for the mean curvature of $\varphi$ and ${\rm Ric}_{\overline M}$ for the Ricci tensor of $\overline M$. Nevertheless, since $H$ is constant and $\overline M$ is Einstein, we get 
$$\mathrm{div}_M(X)=\eta|A|^2.$$
Hence, \eqref{eq:growth of the support function} is equivalent to  {\color{black} $\mathrm{div}_M(X)\geq f$ on $M$}. 

The boundedness of $A$ and the fact that $|Y|=1$ imply the boundedness of $X$. Since $M$ has polynomial volume growth, Theorem \ref{prop:maximum principle II} implies {\color{black} $f\leq 0$ on $M$}. However, since $f\geq 0$, we conclude that {\color{black} $f\equiv 0$ on $M$}; hence, {\color{black} $N\equiv Y$ on $M$}, which is, then, a leaf of $\langle Y\rangle^{\bot}$. 

%As we have already noticed, the leaves of $\langle Y\rangle^{\bot}$ are totally geodesic, hence minimal, in $\overline M$; this implies that $M\setminus\Omega$ is minimal in $\overline M$. But since itself $M$ is connected and of cmc, we conclude that $M$ is minimal in $\overline M$. Finally, since $M$ is connected and coincides with a leaf of $\langle Y\rangle^{\bot}$ in an open set, and both $M$ and this leaf are minimal in $\overline M$, the tangency principle for minimal surfaces assures that $M$ is a leaf of $\langle Y\rangle^{\bot}$.
\end{proof}

For the following corollaries, recall that a {\em Riemannian group} is a Lie group $G$ endowed with a biinvariant metric. In this case, it is a well known fact that the elements of the Lie algebra $\mathfrak g$ of $G$ are Killing vector fields of constant norm, and those in the center of $\mathfrak g$ are parallel. If the biinvariant metric of $G$ is Einstein, then of course we can let $\overline M=G$ in the previous result, thus getting the following

\begin{corollary}\label{coro:Bernstein for CMC hypersurfaces of Lie groups}
Let $G^{n+1}$ be an Einstein Riemannian Lie group with Lie algebra $\mathfrak g$, and $\varphi:M^n\to G^{n+1}$ be a connected, orientable, complete noncompact hypersurface of $G$, oriented by the choice of a unit normal vector field $N$. Assume that $M$ is of cmc and that its second fundamental form $A$ with respect to $N$ is bounded. Assume further that there exists a nontrivial element $Y\in\mathfrak g$ which is either in the center of $\mathfrak g$ or is a canonical direction for $\varphi$. If $M$ has polynomial volume growth and the support function $\langle N,Y\rangle$ {\color{black} satisfies \eqref{eq:growth of the support function} on $M$}, then $M$ is a lateral class of a codimension one Lie subgroup of $G$.
\end{corollary}

\begin{proof}
The previous result assures that $M$ is a leaf of $\langle Y\rangle^{\bot}$ and, as such, is totally geodesic in $G$. Since the distribution $\langle Y\rangle^{\bot}$ is generated by left invariant vector fields and $M$ is a leaf of it, we conclude that $\langle Y\rangle^{\bot}$ is integrable, hence, a codimension one Lie subalgebra of $\mathfrak g$. Therefore, the connectedness of $M$ guarantees that it coincides with a lateral class of a Lie subgroup of $G$.
\end{proof}

The next corollary specializes the former to cmc hypersurfaces of $\mathbb R^{n+1}$. It can be seen as a partial extension of a famous result of Schoen, Simon and Yau (cf. \cite{Schoen:75} or \cite{Xin:03}) to the case of bounded scalar curvature. 

\begin{corollary}\label{coro:Bernstein for CMC hypersurfaces of Rn}
Let $M^n$ be a connected, orientable, complete noncompact Riemannian manifold of bounded scalar curvature and polynomial volume growth. Assume that $\varphi:M\to\mathbb R^{n+1}$ is a cmc immersion, and let $N$ be a unit normal vector field along $M$. If there exists a unit vector $Y\in\mathbb R^{n+1}$ such that the support function $\langle N,Y\rangle$ {\color{black} satisfies \eqref{eq:growth of the support function} on $M$}, then $\varphi(M)$ is a hyperplane orthogonal to $Y$.
\end{corollary}

\begin{proof}
If $A$ stands for the Weingarten operator relative to $N$, then Gauss' equation gives $|A|^2=n^2H^2-n(n-1)R$, where $R$ is the scalar curvature of $M$ and $H$ is the mean curvature of $\varphi$ with respect to $N$. Therefore, the boundedness of $R$ implies that of $A$, and it suffices to apply the previous corollary to $\varphi(M)$.
\end{proof}

We now extend Theorem \ref{thm:Bernstein for CMC hypersurfaces} to the case of higher order mean curvatures, and to this end we need to recall a few facts concerning these objects.

In the sequel, $\varphi:M^n\to M^{n+1}$ stands for an isometric immersion from a connected, orientable Riemannian manifold $M^n$ into an oriented Riemannian manifold $\overline M$. We orient $M$ by the choice of a globally defined unit normal vector field $N$, and let $A$ denote the corresponding second fundamental form. 

Following section 3 of \cite{Alias:06}, one defines the $r$-th Newton transformation $T_r:\mathfrak X(M)\to\mathfrak X(M)$ recursively by letting
$$T_0=I\ \ \text{and}\ \ T_r=S_rI-AT_{r-1},\ \ 1\leq r\leq n,$$
where $I$ denotes the identity in $\mathfrak X(M)$ and $S_r(p)$ the $r$-th elementary symmetric sum of the eigenvalues of $A_p$, for every $p\in M$. 

An easy induction shows that each 
$$T_r=S_rI-S_{r-1}A+\cdots+(-1)^{r-1}S_1A^{r-1}+(-1)^rA^r.$$
In particular, $T_n=p_A(A)$, where $p_A$ is the characteristic polynomial of $A$; hence, $T_n=0$ by Cayley-Hamilton theorem. 

Since $A$ is self adjoint and $T_r$ is a polynomial in $A$, every base which diagonalizes $A$ also diagonalizes $T_r$, and using this fact one can establish, for $1\leq r\leq n$, the standard formulas
\begin{equation}\label{eq:tracos dos Tr}
\begin{split}
&\text{tr}(T_r)=(n-r)S_r,\\
&\text{tr}(AT_r)=(r+1)S_{r+1},\\
&\text{tr}(A^2T_{r-1})=S_1S_r-(r+1)S_{r+1},
\end{split}
\end{equation}
where $\text{tr}(\cdot)$ stands for the trace of the linear operator within parentheses. In particular, $r=1$ in the first and third formulas above yields
$$\text{tr}(T_1)=(n-1)S_1=n(n-1)H\ \ \text{and}\ \ |A|^2=\text{tr}(A^2)=S_1^2-2S_2.$$

Given a Killing vector field $Y$ on $\overline M$, and letting (as before) $Y^{\top}$ denote the orthogonal projection of $Y_{|M}$ onto $TM$, one can compute for $0\leq r\leq n$ (cf. formula (8.4) of \cite{Alias:06})
\begin{equation}\label{eq:divergence of TrYtangent for Y Killing}
\text{div}_M(T_rY^{\top})=\langle\text{div}_MT_r,Y\rangle+\text{tr}(AT_r)\langle N,Y\rangle.
\end{equation}
Here, $\text{div}_MT_r$, the divergence of $T_r$, is the vector field on $M$ defined by
$$\text{div}_MT_r=\text{tr}(\nabla T_r).$$

One can show (cf. Lemma $3.1$ of \cite{Alias:06}, for instance) that, if $\{e_1,\ldots,e_n\}$ is a local orthonormal frame on $M$ and $V\in\mathfrak X(M)$, then
$$\langle\text{div}_MT_r,V\rangle=\sum_{j=1}^r\sum_{i=1}^n\langle\overline R(N,T_{r-j}e_i)e_i,A^{j-1}V\rangle,$$
where $\overline R$ is the curvature operator of $\overline M$. In particular, if $\overline M$ has constant sectional curvature, then this formula readily shows that $\text{div}_MT_r=0$ on $M$.

We now need the following

\begin{lemma}
In the notations above, if $\overline M$ has constant sectional curvature and $Y$ is a Killing vector field on $\overline M$, then, for $1\leq r\leq n$, we have
$$\text{\rm div}_M(AT_{r-1}Y^{\top})=\langle\nabla S_r,Y^{\top}\rangle+\text{\rm tr}(A^2T_{r-1})\langle N,Y\rangle.$$
\end{lemma}

\begin{proof}
Since $AT_{r-1}=S_rI-T_r$, we can compute
\begin{equation}\nonumber
\begin{split}
\text{div}_M(AT_{r-1}Y^{\top})&\,=\text{div}_M(S_rY^{\top})-\text{div}_M(T_rY^{\top})\\
&\,=\langle\nabla S_r,Y^{\top}\rangle+S_r\text{div}_M(Y^{\top})-\text{div}_M(T_rY^{\top})\\
\end{split}
\end{equation}
Now, taking into account that $\text{div}_MT_r=0$ and substituting \eqref{eq:divergence of TrYtangent for Y Killing} and \eqref{eq:tracos dos Tr}, we obtain  
\begin{equation}\nonumber
\begin{split}
\text{div}_M(AT_{r-1}Y^{\top})&\,=\langle\nabla S_r,Y^{\top}\rangle+S_r\text{tr}(A)\langle N,Y\rangle-\text{\rm tr}(AT_r)\langle N,Y\rangle\\
&\,=\langle\nabla S_r,Y^{\top}\rangle+(S_1S_r-(r+1)S_{r+1})\langle N,Y\rangle\\
&\,=\langle\nabla S_r,Y^{\top}\rangle+\text{tr}(A^2T_{r-1})\langle N,Y\rangle.
\end{split}
\end{equation}
\end{proof}

We are now ready to generalize Theorem \ref{thm:Bernstein for CMC hypersurfaces}.

\begin{theorem}\label{thm:Bernstein for Sr constant hypersurfaces}
Let $\overline M^{n+1}$ be an oriented Riemannian manifold with constant sectional curvature, endowed with a Killing vector field $Y$ of unit norm. Let $\varphi:M^n\to M^{n+1}$ be an isometric immersion of a connected, orientable, complete noncompact Riemannian manifold $M^n$ into $\overline M$. Orient $M$ by the choice of a globally defined unit normal vector field $N$, and assume that the corresponding second fundamental form $A$ is bounded. Assume further that $T_{r-1}$ is nonnegative and $\text{\rm tr}(T_r)$ is constant on $M$, for some $1\leq r<n$. If $M$ has polynomial volume growth, $Y$ is either parallel or a canonical direction for $\varphi$ and the support function $\eta=\langle N,Y\rangle$ satisfies
\begin{equation}\label{eq:growth of the support function in the r case}
\eta\geq\frac{1}{\text{\rm tr}(A^2T_{r-1})+1}
\end{equation}
{\color{black} on $M$}, then $M$ is a leaf of the distribution $\langle Y\rangle^{\bot}$. In particular, $M$ is totally geodesic in $\overline M$.
\end{theorem}

\begin{proof}
Once again we let $\eta=\langle N,Y\rangle$ and $f=1-\eta$, so that $f\geq 0$, with equality if and only if $N=Y$ along $M$. We also set $X=AT_{r-1}Y^{\top}$. 

As in the proof of Theorem \ref{thm:Bernstein for CMC hypersurfaces}, if $Y$ is parallel, then $\nabla\eta=-AY^{\top}$. If $Y$ is a canonical direction for $\varphi$, then the fact that $T_{r-1}$ is a polynomial in $A$ assures that $AT_{r-1}Y^{\top}=\mu Y^{\top}$ for some function $\mu$ on $M$; then \eqref{eq:gradient of the support function of a Killing vector field} yields
\begin{equation}\nonumber
\begin{split}
\langle\nabla\eta,X\rangle&\,=\langle-AY^{\top}-\overline\nabla_NY,X\rangle\\
&\,=-\langle AY^{\top},AT_{r-1}Y^{\top}\rangle-\langle\overline\nabla_NY,\mu Y^{\top}\rangle\\
&\,=-\langle A^2T_{r-1}Y^{\top},Y^{\top}\rangle-\mu\langle\overline\nabla_NY,Y^{\top}\rangle\\
&\,=-\langle A^2T_{r-1}Y^{\top},Y^{\top}\rangle,
\end{split}
\end{equation}
where, in the last equality above, we used the fact that $\overline\nabla_NY$ has no orthogonal component and $|Y|=1$. Thus, in each of the cases above, we have
$$\langle\nabla f,X\rangle=-\langle\nabla\eta,X\rangle=\langle A^2T_{r-1}Y^{\top},Y^{\top}\rangle.$$
Since $T_{r-1}$ is nonnegative and self adjoint, it has a square root $Q_{r-1}$ which also commutes with $A$. Hence, $Q_{r-1}$ is self adjoint and the last computation above provides
$$\langle\nabla f,X\rangle=\langle(AQ_{r-1})^2Y^{\top},Y^{\top}\rangle=|AQ_{r-1}Y^{\top}|^2\geq 0.$$

On the other hand, since $S_r$ is constant, the previous lemma gives
$$\text{\rm div}_M(X)=\text{tr}(A^2T_{r-1})\langle N,Y\rangle.$$
Hence, \eqref{eq:growth of the support function in the r case} is equivalent to {\color{black} $\mathrm{div}_M(X)\geq f$ on $M$}.

The boundedness of $A$ on $M$ {\color{black} implies} that of $AT_{r-1}$; this, together with the fact that $|Y|=1$, give the boundedness of $X$ on $M$. Since $M$ has polynomial volume growth, Theorem \ref{prop:maximum principle II} implies {\color{black} $f\leq 0$ on $M$}. The rest of the proof thus goes as in {\color{black} the proof of Theorem \ref{thm:Bernstein for CMC hypersurfaces}}.
\end{proof}

\section{On the existence and size of minimal submanifolds}
\label{section4}

Along all of this section, unless stated otherwise, $\overline M^m$ stands for a Riemannian manifold with metric tensor $\overline g=\langle\cdot,\cdot\rangle$ and Levi-Civita connection $\overline\nabla$. 

We recall that a vector field $Y\in\mathfrak X(\overline M)$ is conformal with conformal factor $\phi\in\mathcal C^{\infty}(\overline M)$ provided $\mathcal L_Y\overline g=2\phi\overline g$, where $\mathcal L_Y$ stands for the Lie derivative in the direction of $Y$. If this is so, it is straightforward to verify that $\langle\overline\nabla_XY,Z\rangle=\phi\langle Z,Z\rangle$, {\color{black} for all} $Z\in\mathfrak X(\overline M)$. Then, the divergence of $Y$ on $\overline M$ is given by $\text{div}_{\overline M}(Y)=m\phi$. 

We shall need the following

\begin{lemma}\label{lemma:relation between divergences}
Let $\overline M^m$ be a Riemannian manifold with metric tensor $\overline g=\langle\cdot,\cdot\rangle$ and $Y\in\mathfrak X(\overline M)$ be a conformal vector field with conformal factor $\phi$. If $\varphi:M^n\to\overline M^m$ is an isometric immersion and $X=(Y_{|M})^{\top}$ is the orthogonal projection of $Y_{|M}$ into $TM$, then
\begin{equation}\label{eq:relation between divergences}
\text{\rm div}_M(X)=n\big(\phi_{|M}+\langle Y_{|M},\overrightarrow{H}\rangle\big),
\end{equation}
where $\overrightarrow{H}$ stands for the mean curvature vector of $\varphi$.
\end{lemma}

\begin{proof}
We fix a local orthonormal frame field $(e_1,\ldots,e_n)$ on an open set $U\subset M$. Setting $l=m-n$ and shrinking $U$, if necessary, we can also consider an orthonormal frame field $(N_1,\ldots,N_l)$ for $TU^{\bot}$. 

Writing $Y$ instead of $Y_{|M}$ for the sake of simplicity, we have $X=Y-\sum_{j=1}^l\langle Y,N_j\rangle N_j$ on $U$. Letting $\nabla$ stand for the Levi-Civita connection of $M$, we get on $U$ that
\begin{equation}\nonumber
\begin{split}
\text{div}_M(X)&\,=\sum_{i=1}^n\langle\nabla_{e_i}X,e_i\rangle=\sum_{i=1}^n\langle\overline\nabla_{e_i}X,e_i\rangle\\
&\,=\sum_{i=1}^n\langle\overline\nabla_{e_i}Y,e_i\rangle-\sum_{i=1}^n\sum_{j=1}^l\langle\overline\nabla_{e_i}\langle Y,N_j\rangle N_j,e_i\rangle\\
&\,=n\phi-\sum_{j=1}^l\langle Y,N_j\rangle\sum_{i=1}^n\langle\overline\nabla_{e_i}N_j,e_i\rangle.
\end{split}
\end{equation}

Denoting by $A_j$ the Weingarten operator of $\varphi$ in the direction of $N_j$, we can continue the computation above by writting
\begin{equation}\nonumber
\begin{split}
\text{div}_M(X)&\,=n\phi+\sum_{j=1}^l\langle Y,N_j\rangle\sum_{i=1}^n\langle A_je_i,e_i\rangle\\
&\,=n\phi+\sum_{j=1}^l\langle Y,N_j\rangle\text{tr}(A_j),
\end{split}
\end{equation}
where $\text{tr}(\cdot)$ stands for the trace of the operator within parentheses. Therefore,
$$\text{div}_M(X)=n\phi+\langle Y,\sum_{j=1}^l\text{tr}(A_j)N_j\rangle=n\phi+\langle Y,n\overrightarrow{H}\rangle,$$
as wished.
\end{proof}

In the sequel, if $B^k$ is another Riemannian manifold, $\pi:\overline M^m\to B^k$ is a Riemannian submersion and $Z\in\mathfrak X(B)$, we let $\tilde Z$ denote the horizontal lift of $Z$ to $\overline M$. Also, given a smooth function $h:B\to\mathbb R$, we shall write $\tilde h$ to denote the composition $\tilde h=h\circ\pi:\overline M\to\mathbb R$. Letting $Dh$ and $\overline\nabla\tilde h$ denote the gradients of $h$ (on $B$) and $\tilde h$ (on $\overline M$), respectively, we have $\overline\nabla\tilde h=\widetilde{Dh}$.

If, in addition, $\varphi:M^n\to\overline M^m$ is an isometric immersion, then $\varphi$ can be locally seen as the inclusion. Therefore, if $h$ is as above and there is no danger of confusion, we set $f=\tilde h\circ\varphi:=\tilde h_{|M}$ (look at the diagram below).
$$\xymatrix{
M^n\ar[r]^{\varphi}&\overline M^m\ar[d]^{\pi}\ar[rd]^{\tilde h}&\\
&B^k\ar[r]_h&\mathbb R
}$$
If $\nabla f$ denotes the gradient of $f$ on $M$, then 
\begin{equation}\label{eq:the gradient of f}
\nabla f=\big(\overline\nabla\tilde h\big)^{\top}=\big(\widetilde{Dh}\big)^{\top},
\end{equation}
the orthogonal projection, onto $TM$, of the restriction of $\widetilde{Dh}$ to $T\overline M_{|M}$.

We now assume that there exist a smooth function $g:B\to[0,+\infty)$ and a vector field $Y\in\mathfrak X(\overline M)$ such that $\overline\nabla\tilde h=\tilde gY$. Then, with $f$ as above and $X=(Y_{|M})^{\top}$, \eqref{eq:the gradient of f} gives $\nabla f=\tilde g_{|M}X$, whence
\begin{equation}\label{eq:nonnegativity of the inner product}
\langle\nabla f,X\rangle=\tilde g_{|M}|X|^2\geq 0.
\end{equation}
Moreover, if $Y$ is conformal with conformal factor $\phi$ and $a$ is a positive constant, then it follows from \eqref{eq:relation between divergences} that 
\begin{equation}\label{eq:comparing divergence of X with af}
\text{div}_M(X)\geq af\Leftrightarrow n\big(\phi_{|M}+\langle Y_{|M},\overrightarrow{H}\rangle\big)\geq a\tilde h_{|M}.
\end{equation}
In particular, this is automatically true if $\varphi$ is minimal and $n\phi\geq a\tilde h$ along $\varphi(M)$.

We can now state and prove our main results.

\begin{theorem}\label{thm:general on submersions}
Let $\pi:\overline M^m\to B^k$ be a Riemannian submersion and $Y\in\mathfrak X(\overline M)$ be conformal with conformal factor $\phi$. Assume that there exist smooth functions $g:B\to\mathbb R$ and $h:B\to(0,+\infty)$ such that $\overline\nabla\tilde h=\tilde gY$. Let $\varphi:M^n\to\overline M^m$ be an isometric immersion from an oriented, complete noncompact Riemannian manifold $M$ into $\overline M$, such that $\tilde g\geq 0$, $n\phi\geq a\tilde h$ and $|Y|\leq c$ on $\varphi(M)$, for some positive constants $a$ and $c$. 
\begin{enumerate}[$(a)$]
\item If $M$ has polynomial volume growth, then $\varphi$ cannot be minimal.
\item If $M$ has exponential volume growth, say like $e^{\beta t}$, and $\varphi$ is minimal, then $\tilde h\leq\frac{c\beta}{a}$ on $\varphi(M)$.
\end{enumerate}
\end{theorem}

\begin{proof}
We have done almost all of the work along the previous discussion: with $f=\tilde h_{|M}$ and $X=(Y_{|M})^{\top}$, we have $|X|\leq|Y_{|M}|\leq c$ and, from \eqref{eq:nonnegativity of the inner product}, $\langle\nabla f,X\rangle\geq 0$ on $M$. Moreover, if $\varphi$ is minimal, then \eqref{eq:comparing divergence of X with af} and our hypotheses give $\text{div}_M(X)\geq af$ on all of $M$. We now consider cases (a) and (b) separately:\\

\noindent (a) Theorem \ref{prop:maximum principle II} ascertains that $f\leq 0$ along $M$, which is a contradiction, for $h$ is positive on $B$.\\

\noindent (b) This follows immediately from Theorem \ref{prop:maximum principle II}.
\end{proof}

We now specialize the previous result in the following way: we let $\Sigma^{m-1}$ be a Riemannian manifold with metric $\sigma$ and $I\subset\mathbb R$ be an open interval with its standard metric $dt^2$. We set $\overline M^m=\Sigma^{m-1}\times I$ and let $\pi_{\sigma}:\overline M\to\Sigma$ and $\pi_I:\overline M\to I$ denote the standard projections. If $h:I\to(0,+\infty)$ is a smooth function and $\overline g=\tilde h^2\pi_{\Sigma}^*\sigma+\pi_I^*dt^2$, then $\overline g$ is a metric tensor on $\overline M$, with respect to which $\overline M$ is said to be the warped product of $\Sigma$ and $I$, with warping function $h$. We summarize the above by writing $\overline M=\Sigma\times_hI$, and note that $\pi_I:\overline M\to I$ is a Riemannian submersion.

It is a standard fact that $Y=\tilde h\tilde{\partial_t}$ is a conformal vector field with conformal factor $\phi=\tilde{h'}$. Moreover, if $g=\frac{h'}{h}$, then
\begin{equation}\label{eq:relation between gradient of h tilde and Y in warpeds}
\overline\nabla\tilde h=\tilde{h'}\tilde\partial_t=\tilde g\tilde h\tilde\partial_t=\tilde gY.
\end{equation}

We are thus left with the following

\begin{corollary}\label{cor:specializing to warped products}
Let $\overline M=\Sigma\times_hI$ be a warped product as above. Let $\varphi:M^n\to\overline M^m$ be an isometric immersion from an oriented, complete noncompact Riemannian manifold $M$ into $\overline M$, such that $\tilde h\leq\frac{n}{a}\tilde h'$ and $\tilde h\leq c$ on $\varphi(M)$, for some positive constants $a$ and $c$. 
\begin{enumerate}[$(a)$]
\item If $M$ has polynomial volume growth, then $\varphi$ cannot be minimal.
\item If $M$ has exponential volume growth, say like $e^{\beta t}$, and $\varphi$ is minimal, then $\beta\geq a$.
\end{enumerate}
\end{corollary}

\begin{proof}
Setting $Y=\tilde h\tilde{\partial_t}$ and $g=\frac{h'}{h}$, we already know that $Y$ is conformal, with conformal factor $\phi=\tilde h'$, and \eqref{eq:relation between gradient of h tilde and Y in warpeds} gives $\overline\nabla\tilde h=\tilde gY$. Moreover, our hypotheses assure that $\tilde g\geq\frac{a}{n}>0$, $n\phi-a\tilde h\geq 0$ and $|Y|\leq c$ along $\varphi(M)$. Item (a) is now a particular case of the previous result. 

As for item (b), we conclude from the previous result that $\tilde h\leq\frac{c\beta}{a}$ on $\varphi(M)$. Since $\tilde h>0$, one has to have $\beta>0$. We can then apply item (b) again, this time with $\frac{\beta c}{a}$ in place of $c$, to conclude that $\tilde h\leq c\big(\frac{\beta}{a}\big)^2$ on $\varphi(M)$. By iterating this reasoning, we therefore conclude that $\tilde h\leq c\Big(\frac{\beta}{a}\Big)^l$ on $\varphi(M)$, for every $l\geq 1$. If $0<\beta<a$, then $0<\frac{\beta}{a}<1$ and, letting $l\to+\infty$, we conclude that $\tilde h\leq 0$ on $\varphi(M)$, which is impossible. Hence, $\beta\geq a$.
\end{proof}

We close this section with the following interesting particular cases of the previous result.

\begin{corollary}\label{cor:specializing to Rm}
Let $\varphi:M^n\to\mathbb R^m$ be an isometric immersion from an oriented, complete noncompact Riemannian manifold $M$ into the Euclidean $m$-space, such that $\varphi(M)\subset B_{\mathbb R^m}(0,R)$, for some $R>0$.
\begin{enumerate}[$(a)$]
\item If $M$ has polynomial volume growth, then $\varphi$ cannot be minimal.
\item If $M$ has exponential volume growth, say like $e^{\beta t}$, and $\varphi$ is minimal, then $R\geq\frac{n}{\beta}$.
\end{enumerate}
\end{corollary}

\begin{proof}
Taking any $S>R$ and composing $\varphi$ with a translation (which depends on the chosen $S$), we can assume that $$\varphi(M)\subset B_{\mathbb R^m}(0,S)$$ and $0\notin\varphi(M)$. We now look at $\mathbb R^m\setminus\{0\}$ as the warped product
$$\mathbb R^m\setminus\{0\}=\mathbb S^{m-1}\times_t(0,+\infty),$$
where $t$ is the standard coordinate function on $(0,+\infty)$.

Then, in the notations of the statement of the previous corollary and with $a=\frac{n}{S}$ and $c=S$, we have $\tilde h\leq\frac{n}{a}\tilde h'$ and $\tilde h\leq c$. Thus, the former corollary gives item (a), as well as, in item (b), $\beta\geq\frac{n}{S}$. However, this is the same as $S\geq\frac{n}{\beta}$, and since this holds for every $S>R$, we get $R\geq\frac{n}{\beta}$.
\end{proof}

\begin{corollary}\label{cor:specializing to Hm}
Let $\varphi:M^n\to\mathbb H^m$ be an isometric immersion from an oriented, complete noncompact Riemannian manifold $M$ into the hyperbolic $m$-space, such that $\varphi(M)$ is contained in the open half space bounded by a horosphere.
\begin{enumerate}[$(a)$]
\item If $M$ has polynomial volume growth, then $\varphi$ cannot be minimal.
\item If $M$ has exponential volume growth, say like $e^{\beta t}$, and $\varphi$ is minimal, then $\beta\geq n$.
\end{enumerate}
\end{corollary}

\begin{proof}
We look at the hyperbolic $m$-space as the warped product
$$\mathbb H^m=\mathbb R^{m-1}\times_{e^{t}}\mathbb R,$$
where $t$ is the standard coordinate function on $\mathbb R$ and the $t$-slices $\mathbb R^{m-1}\times\{t\}$ are horospheres. Then $h(t)=h'(t)=e^t$ and, in the notations of the statement of the Corollary  \ref{cor:specializing to warped products}, we have $\tilde h\leq\frac{n}{a}\tilde h'$ on $\varphi(M)$ with $a=n$. Moreoever, the condition that $\varphi(M)$ is contained in the open half space bounded by a horosphere is equivalent to the fact that $\tilde h\leq c$ on $\varphi(M)$ for certain positive  $c$. Thus, the result follows directly from Corollary \ref{cor:specializing to warped products}.
\end{proof}

\end{document}